\documentclass[a4paper]{amsart}

\usepackage{import}

\usepackage{comment}

\usepackage{amsmath,amsfonts,amssymb,amsthm,bm,xcolor,yfonts,extpfeil,graphicx,verbatim,mathrsfs}

\usepackage[pdftex,linktocpage]{hyperref}
\usepackage{smartref}

\newcommand{\NN}{\mathbb N}

\newcommand{\RR}{\mathbb R}

\newcommand{\kk}{\Bbbk}
\newcommand{\DD}{\mathscr{D}}

\newcommand{\im}{\mathop{\mathrm{im}}}
\newcommand{\rann}{\text{r.ann}_{A_2}}
\DeclareMathOperator{\Span}{span}

\newcommand{\NX}{\widetilde{X}}

\newcommand{\dx}{\partial_{x}}
\newcommand{\dy}{\partial_{y}}

\DeclareMathOperator{\Spec}{Spec}

\newcommand{\If}{\mathbb{I}_{A_2}(fA_2)}

\theoremstyle{plain}
\newtheorem{theorem}{Theorem}[section]
\newtheorem{proposition}[theorem]{Proposition}
\newtheorem{corollary}[theorem]{Corollary}
\newtheorem{lemma}[theorem]{Lemma}
\newtheorem{question}[theorem]{Question}

\theoremstyle{definition}
\newtheorem{Definition}[theorem]{Definition}
\newtheorem{notation}[theorem]{Notation}

\newtheorem{example}[theorem]{Example}

\DeclareMathOperator{\Hom}{Hom}
\DeclareMathOperator{\End}{End}

\DeclareMathOperator{\gdim}{gldim}

\DeclareMathOperator{\lt}{lt}

\DeclareMathOperator{\supp}{supp}
\DeclareMathOperator{\ann}{ann}
\DeclareMathOperator{\GKdim}{GKdim}
\DeclareMathOperator{\kdim}{Kdim}

\makeatletter \def\@cite#1#2{{\normalfont[{#1}\if@tempswa , #2\fi]}} \makeatother

\begin{document}
\title{Idealizers in the Second Weyl Algebra}
\author{Ruth Reynolds}
\date{\today}
\thanks{The University of Edinburgh, James Clerk Maxwell Building, Peter Guthrie Tait Road, Edinburgh EH9 3FD.
 { \tt r.a.e.reynolds@sms.ed.ac.uk}}
\maketitle
\begin{abstract}
    Given a right ideal $I$ in a ring $R$, the idealizer of $I$ in $R$ is the largest subring of $R$ in which $I$ becomes a two-sided ideal. In this paper we consider idealizers in the second Weyl algebra $A_2$, which is the ring of differential operators on $\kk[x,y]$ (in characteristic $0$). Specifically, let $f$ be a polynomial in $x$ and $y$ which defines an irreducible curve whose singularities are all cusps. We show that the idealizer of the right ideal $fA_2$ in $A_2$ is always left and right noetherian, extending the work of McCaffrey.
\end{abstract}
\tableofcontents
\section{Introduction}
Let $\kk$ be an algebraically closed field of characteristic $0$. Throughout this paper, by `a variety' we mean `an irreducible affine algebraic variety over $\kk$' and by `a curve' we mean an `irreducible affine algebraic curve over $\kk$'. We denote the ring of regular functions on a variety $X$ by $\mathcal{O}(X)$. In this paper we are interested in idealizers in rings of differential operators. In particular we study $\If\subseteq A_2$ where $f\in \kk[x,y]$ defines a curve $X$ such that the normalisation map $\phi: \widetilde{X}\to X$ is injective.

The noetherianity of idealizer subrings in Weyl algebras and more general rings of differential operators has been studied before. Robson was the first to study idealizer subrings in $A_1$, the first Weyl algebra. Theorem \ref{A_1} is a special case of a general result about idealizers in HNP rings, of which $A_1$ is an example.
\begin{theorem}\emph{\cite[Theorem $7.4$]{Robson1972}}\label{A_1}
Let $I\leq_rA_1$ be nonzero. Then $\mathbb{I}_{A_1}(I)$ is a left and right noetherian ring.
\end{theorem}
However, the higher Weyl algebras are no longer hereditary (in fact $\gdim A_n = n$) so alternative techniques must be used to study idealizers in these rings. Despite this, it turns out that idealizers at certain types of right ideal in higher Weyl algebras behave very well, as we see in the following theorem.
\begin{theorem}\emph{\cite[Proposition $2.3$]{Resco1980}}\label{Resco}
View $A_1\subseteq A_n$, the $n^{th}$ Weyl algebra. Let $I\leq_r A_1$ be a right ideal in the first Weyl algebra. Then $\mathbb{I}_{A_n}(IA_n)$ is right and left noetherian.
\end{theorem}
\begin{proof}
We simply observe that 
\begin{align*}
    \mathbb{I}_{A_n}(IA_n) &\cong \mathbb{I}_{A_1}(I)\otimes_{\kk}A_{n-1}\\
    &\cong \mathbb{I}_{A_1}(I)[x_2,\dots,x_n][\partial_2;\partial/\partial x_2]\dots [\partial_n;\partial/\partial x_n]
\end{align*}
and hence, by Theorem \ref{A_1}, the noetherianity of $\mathbb{I}_{A_n}(IA_n)$ is induced from that of $\mathbb{I}_{A_1}(I)$ by viewing it as an Ore extension of $\mathbb{I}_{A_1}(I)$.
\end{proof}
The situation is not so nice for more complicated ideals; for instance we cannot use the same trick on a right ideal of the form $xA_2+yA_2$ where we view $A_2$ as $\kk[x,y][\dx,\dy]$. Indeed, Resco proved that the conclusion of the previous theorem does not hold for the idealizer associated to this ideal .
\begin{theorem}\emph{\cite[Theorem $2$]{Resco1983}}
The idealizer $\mathbb{I}_{A_2}(xA_2+yA_2)$ is right but not left noetherian.
\end{theorem}

On the other hand, Theorem \ref{Resco} shows that $\If$ is right and left noetherian when $f$ is of the form $ax+by+c$ for some $a,b,c \in \kk$. Hence, it is natural to ask the following question.
\begin{question}
For which $f\in \kk[x,y]$ is $\If$ right and left noetherian?
\end{question}
McCaffrey \cite{Mccaffrey} studied these idealizers in the case where $f$ defines a nonsingular curve and obtained the following theorem which we paraphrase.

\begin{theorem}\emph{\cite{Mccaffrey}}
Let $f\in \kk[x,y]$ define a nonsingular curve. Then the idealizer ring $\If$ is right noetherian.
\end{theorem}
In this paper we strengthen and generalise this result to a class of singular curves.
\begin{theorem}\label{fdeets}
Let $X = \mathbb{V}(f)$ be a plane curve such that the normalisation map $\phi:\NX \to X$ is injective. Then $\If$ is right and left noetherian.
\end{theorem}
We reduce this problem to considering the noetherianity of $\Hom_{A_2}(A_2/fA_2,S)$ as a right $\If/fA_2$-module, where $S$ is a simple right $A_2$-module. We then split this into two cases: when $S \cong A_2/mA_2$ where $m\lhd \kk[x,y]$ is maximal, or when $S$ is not of this form. For the first case, we obtain that the module is noetherian by a careful combinatorial argument. For the second case, we use a localisation argument and the work of \cite{SMST} to derive the result as a consequence of Bernstein's preservation of holonomicity.

A straightforward application of Theorems \ref{A_1} and \ref{Resco} shows that $\mathbb{I}_{A_2}(xA_2)$ is right noetherian. The techniques McCaffrey develops for his result are underpinned by the fact that, locally, nonsingular curves look like the affine line. Unfortunately, this is not the case with singular curves and so we must use different techniques. 

The reason that we require the normalisation map $\phi:\NX \to X$ to be injective is to do with the structure of the ring of differential operators on cuspidal curves in comparison to more complicated curves.
The ring of global differential operators over a variety $X$, $\DD(X)$, in the sense of Grothendieck \cite[16.8.1]{Groeth}, has many nice properties when $X$ is nonsingular. In particular, $\DD(X)$ is a finitely generated, noetherian $\kk$-algebra and, when $X$ is a curve, $\DD(X)$ is a simple, hereditary, noetherian, prime (HNP) ring \cite[16.11.2]{Groeth}, \cite[Chapter 3 Theorem 2.5]{bjorkdiff}. Smith and Stafford further showed that if $X$ is a singular curve then $\DD(X)$ is still a finitely generated, noetherian $\kk$-algebra and, if the normalisation map $\phi:\widetilde{X}\to X$ is injective, $\DD(X)$ remains a simple hereditary ring,  \cite{SMST}.

We also have the following link between rings of differential operators and idealizers.

\begin{proposition}\emph{\cite[Proposition $1.6$]{SMST}}\label{idealizer-diffops}
Let $Y$ be a nonsingular variety and $X$ a subvariety defined by an ideal $I$ of $\mathcal{O}(Y)$. Then
$$\DD(X) \cong \frac{\mathbb{I}_{\DD(Y)}(I\DD(Y))}{I\DD(Y)}.$$
\end{proposition}
It is natural to ask, given the strong relation between idealizers and rings of differential operators, whether idealizers at singular curves are also left and right noetherian. In this paper we answer that question in the affirmative, at least when the curve is suitably well-behaved. 

\textbf{Acknowledgements.}
The author is an EPSRC-funded student at the University of Edinburgh,
and the material contained in this paper forms part of her PhD thesis. The author would like to thank her supervisor Susan J. Sierra for suggesting this problem and providing guidance, and also the EPSRC.

\section{Preliminaries}
We shall start with a summary of the definitions and results from the literature concerning idealizers at smooth curves and rings of differential operators of possibly singular curves.

We begin by stating important results from the literature which will be used in the main part of the paper.

\subsection{Rings of Differential Operators of Curves}\label{smoothfacts}
We begin with the case when $X$ is a nonsingular variety. Recall that we write $\DD(X)$ for $\DD(\mathcal{O}(X))$. Then $\DD(X)$ has the following properties, \cite[Section $16$]{Groeth}:
\begin{enumerate}
    \item[(a)] $\DD(X)$ is a finitely generated, simple, noetherian domain.
    \item[(b)] The global dimension of $\DD(X)$ is finite; more precisely $\gdim(\DD(X)) = \dim X$. Further, $\GKdim(\DD(X)) = 2\dim X$. We see that in the case that $X$ is a curve, $\DD(X)$ is an HNP ring with GK dimension $2$.
\end{enumerate}
\begin{lemma}\emph{\cite[Proposition $1.9$]{Muh}}\label{LcDiffs}
Let $S$ be a multiplicatively closed set in a finitely generated commutative $\kk$-algebra $R$. Then
$$\DD(RS^{-1}) = \DD(R)\otimes_{R}RS^{-1} = \DD(R)S^{-1},$$
that is to say, localising commutes with taking rings of differential operators.
\end{lemma}
We now move on to the case where $X$ is a curve with singular points. An important construction  is that of the normalisation of $X$, which we denote $\widetilde{X}$; that is, $\mathcal{O}(\NX)$ is the integral closure of $\mathcal{O}(X)$ inside the function field $\kk(X)$. When $X$ is a curve, the normalisation is also a curve. We suppose further that the normalisation map $\phi:\widetilde{X}\to X$ is injective; that is, the singularities of $X$ are all cusps.

The following are two sets of differential operators between different modules or $\kk$-algebras.
\begin{Definition}\label{D-modules}
Let $A$ be a commutative $\kk$-algebra and let $M$ and $N$ be $A$-modules. Then we define the space of \emph{$\kk$-linear differential operators from $M$  to $N$ of order at most $n$} inductively by $\DD^{-1}(M,N) = 0$ and for $n\geq 0$:
$$\DD^n(M,N) = \{\theta \in \Hom_{\kk}(M,N) \mid [\theta,a]\in \DD^{n-1}(M,N) \text{ for all } a \in A\}.$$
We denote the space of all $\kk$-linear differential operators from $M$ to $N$ by $\DD_A(M,N)$.
\end{Definition}
We also have the following definition for differential operators between $\kk$-algebras.

\begin{Definition}\label{D-rings}
If $A\subseteq B$ are commutative $k$-algebras then we write 
$$\DD(B,A) = \{D\in \DD(B) \mid D * f \in A \textrm{ for all } f \in B\},$$
where $D*f$ denotes the action of the differential operator $D$ on $f$; that is, the (left) $\DD(B)$-module structure on $B$.
\end{Definition}
We note that $\DD_A(B,A)$ and $\DD(B,A)$ are not necessarily equal. However, in the case that $A$ and $B$ are both domains such that $A\subseteq B\subseteq \textrm{ Fract }(A)$ (the situation in which we are interested), then these two sets are equal \cite[Lemma 2.7]{SMST}. When $B = \mathcal{O}(\widetilde{X})$ and $A = \mathcal{O}(X)$, we shall write $\DD(\widetilde{X},X) \coloneqq \DD(\mathcal{O}(\widetilde{X}),\mathcal{O}(X))$. We note that this is both a right ideal of $\DD(\widetilde{X})$ and also a left ideal of $\DD(X)$. If $\DD(X)$ and $\DD(\NX)$ are Morita equivalent we may use Definition \ref{D-modules} to define $\DD(X,\NX) \cong \DD_{\mathcal{O}(X))}(\mathcal{O}(X),\mathcal{O}(\NX))$, which is $\Hom_{\DD(\NX)}(\DD(\NX,X),\DD(\NX))$, the dual of $\DD(\NX,X)$ \cite[Proposition 3.14]{SMST}.

An important technique in this paper will be to identify $\DD(X)$ and $\DD(\NX)$ with subalgebras of $\DD(K)$ where $K$ is the fraction field of $\mathcal{O}(X)$. If $K = \kk(X)$ is the field of fractions associated to $\mathcal{O}(X)$, then we may identify $\DD(\NX)$ with its image in $\DD(K)$ as follows:
$$\DD(\NX) = \{D \in \DD(K) \mid D*(\mathcal{O}(\NX))\subseteq \mathcal{O}(\NX)\}.$$
We may also identify $\DD(\NX,X)$ as:
$$\DD(\NX,X) = \{D \in \DD(K) \mid D*(\mathcal{O}(\NX))\subseteq \mathcal{O}(X)\}.$$

Then we have the following result due to Smith and Stafford.
\begin{theorem}\emph{\cite[Theorem 3.4]{SMST}}\label{SS}
Let $X$ be a curve. Then the following are equivalent:
\begin{enumerate}
    \item [$(1)$] The normalisation map $\phi: \widetilde{X}\to X$ is injective;
    \item [$(2)$] $\DD(\NX)$ is Morita equivalent to $\DD(X)$ via the progenerator $\DD(\NX,X)$.
    \end{enumerate}

\end{theorem}
From this result we have the following properties of $\DD(X)$.
\begin{corollary}\label{noethsimp}\emph{\cite[Theorems A and B]{SMST}}
Suppose $\phi:\NX \to X$ is injective. Then $\DD(X)$ is
\begin{enumerate}
    \item [(a)] right and left noetherian;
    \item[(b)]  a finitely generated $\kk$-algebra;
    \item[(c)]  a hereditary ring with (Gabriel-Rentschler) Krull dimension $1$, and GK dimension $2$;
    \item[(d)]  a simple ring.
\end{enumerate}
\end{corollary}
\subsection{Holonomic Modules}
We are particularly interested in a certain type of module over a ring of differential operators - a holonomic module. Whilst these modules have a simple definition, their properties are surprisingly good. We have the following result due to Bernstein.
\begin{theorem}\emph{\cite{bernstein1972analytic}}
If $X = \Spec R$ is a nonsingular variety and  $M$ is a nonzero $\DD(R)$-module, then $\GKdim(M)\geq \dim X$.
\end{theorem}

\begin{Definition}
We define a \emph{holonomic} $D$-module to be a finitely generated module such that $\GKdim(M) = \dim X$.
\end{Definition}
The following theorem shows that, in the setting in which we are working, holomonic modules behave particularly well.
\begin{theorem}[Bernstein's preservation of holonomicity]\label{Bern}\emph{\cite[Theorem A]{B}}
Let $D = \DD(X)$ be a ring of differential operators over a smooth algebraic variety $X$. Let $S\subseteq \mathcal{O}(X)$ be a multiplicatively closed subset and let $M$ be a holonomic right $DS^{-1}$-module. Then $M$ is holonomic as a right $D$-module.
\end{theorem}
Although Bernstein gives this result as a statement about derived categories of holonomic $D$-modules, it is well known that open immersions send holonomic $D$-modules to holonomic $D$-modules; for a proof, see \cite[Theorem 3.23]{Elliott}.
\subsection{Idealizers at curves}
In this subsection we summarise the results from the literature which we will use in the main part of the paper. We will focus on the idealizer at a curve defined by a polynomial $f\in \kk[x,y]$, more precisely, we mean the ring:
$$\If = \{P\in A_2 \mid PfA_2\subseteq fA_2\}.$$
Recall the definition of the colon ideal for two right ideals $I,J$ in a ring $R$:
$$(J:I) = \{r\in R \mid rI\subseteq J\}.$$
We have the following result which we will use frequently in this paper.
\begin{proposition}\label{homsquotient}\emph{\cite[Proposition 1.1]{Robson1972}}
Let $I$ be a right ideal in a ring $R$ and let $J$ be a right ideal of $R$ which contains $I$. Then
$$\Hom_R(R/I,R/J) \cong \frac{(J:I)}{J},$$
considered as right modules over $\mathbb{I}_R(I)$.
\end{proposition}
\begin{proposition}\emph{\cite[cf. Prop. $15.5.9$]{MR}}
Let $I,J\lhd R = \kk[x,y]$ and let $\beta \in A_2$. Then
\begin{enumerate}
    \item[(1)] $\beta *R\subseteq I$ if and only if $\beta \in IA_2$;
    \item[(2)] $\beta*I\subseteq J$ if and only if $\beta \in (JA_2:IA_2),$
\end{enumerate}
where $\beta*R$ denotes the standard action of $A_2$ on $R$.
\end{proposition}
\begin{proof}
For $(1)$, this is precisely \cite[Proposition 15.5.9 (i)]{MR}.

\noindent
For $(2)$ we start with the converse. Let $a\in I$, then $a \in IA_2$. Hence 
    $$\beta * a = (\beta a)* 1_R \in (JA_2)*1_R\subseteq JR = J.$$
Now for the forward direction, let $a \in I$ and $r \in R$. Then
$$(\beta a) *r = \beta*ar\in J$$
and hence $(\beta a)* R \subseteq J$. Thus, by $(1)$, $\beta a \in JA_2$, which implies $\beta \in (JA_2:IA_2)$.
\end{proof}
\begin{corollary}\label{idealizer-presentation}
Let $f,g \in \kk[x,y]$. Then 
$$\DD((f),(g)) = \{\theta \in A_2 \mid \theta *(f)\subseteq (g)\} = (gA_2:fA_2)$$
and setting $f=g$ we obtain
$$\If = \{\theta \in A_2\mid \theta * p\in (f) \textrm{ for all } p \in (f)\} = \DD((f),(f)).$$\qed
\end{corollary}
We have the following result about the noetherianity of idealizers:
\begin{proposition}\emph{\cite[Proposition 2.1]{Rog}}\label{homs}
Let $I$ be a right ideal in a noetherian domain $B$. Let $R = \mathbb{I}_B(I)$. Then the following are equivalent:
\begin{enumerate}
    \item[(1)] $R$ is right noetherian.
    \item[(2)] For every right ideal $J\supseteq I$ of $B$, $\Hom_B(B/I,B/J)$ is a right noetherian $R$-module (or $R/I$-module).
\end{enumerate}

\end{proposition}
We next summarise the work of McCaffrey on idealizers at nonsingular curves \cite{Mccaffrey}.
\begin{theorem}\label{McCfnoeth}\emph{[McCaffrey]}
Let $B$ be a regular $\kk$-affine domain and $Q$ a prime ideal in $B$ such that $B/Q$ is regular. Then $\mathbb{I}_{\DD(B)}(Q\DD(B))$ is right noetherian.
\end{theorem}
Further:
\begin{proposition}\label{strucM}\emph{[McCaffrey]}
Let $B$ and $Q$ satisfy the assumptions of Theorem \ref{McCfnoeth}. For $J$ any right ideal of $\DD(B)$ strictly containing $Q$, we have 
$$(J:Q\DD(B)) = \mathbb{I}_{\DD(B)}(Q\DD(B))+J$$
and further, using the identification in Proposition \ref{homsquotient},
$$\frac{(J:Q\DD(B))}{J} = \Hom_{\DD(B)}\left(\frac{\DD(B)}{Q\DD(B)},\frac{\DD(B)}{J}\right)\cong \frac{\mathbb{I}_{\DD(B)}(Q\DD(B))+J}{J}.$$
\end{proposition}
\begin{proposition}\label{strictc}\emph{[McCaffrey]}
Let $B$ and $Q$ satisfy the assumptions of Theorem \ref{McCfnoeth}. If $J$ is any right ideal of $\DD(B)$ strictly containing $Q$, then $$Q\subsetneqq J\cap \mathbb{I}_{\DD(B)}(Q\DD(B)).$$
\end{proposition}
\subsection{Hereditary Rings}
As observed earlier, hereditary rings have particularly nice properties, two of which we detail below.
\begin{lemma}\label{LHisH}
Let $S$ be a localisation of a hereditary ring. If $S$ is not semisimple, then it is also hereditary.
\end{lemma}
\begin{proof}
We simply note that the global dimension of the localisation is bounded above by that of the original ring \cite[7.4.3]{MR}, namely $1$.
\end{proof}
We have the following result which shows that simple modules over localisations of HNP rings are very well behaved. 
\begin{proposition}\label{GOOD}
Let $R$ be an HNP ring and let $S$ a right denominator set such that $RS^{-1}$ is not the full Goldie quotient ring of $R$. Let $M$ be a simple right $RS^{-1}$-module. Then there exists a simple right  $R$-module, $N$, such that
$$M \cong N\otimes_{R}RS^{-1}.$$

\end{proposition}
\begin{proof}
Let $T \coloneqq RS^{-1}$. Note that $M_T$ is either Goldie torsion or Goldie torsionfree. If $M$ were Goldie torsionfree then, by \cite[Lemma 7.17]{GW}, $T$ would have nonzero socle and hence, by \cite[Theorem 7.15]{GW}, $T$ is semisimple; this is impossible. Therefore $M_T$ is (Goldie) torsion and thus so is $M_R$. Consider $mR \leq M$ where $m\in M$ is nonzero. Since $m$ is torsion, $\text{r.ann}_R(mR)\neq 0$, and so $mR \cong R/\text{r.ann}_R(mR)$ has finite length as $R$ is HNP \cite[Lemma 6.2.8]{MR}. Hence $ mR \subseteq M$ contains a simple right $R$-module $N$. Then note that
$$N\otimes_R T \cong NT \subseteq M$$
as $T$ is flat over $R$. By the simplicity of $M$, $NT = M$.
\end{proof}

\section{\texorpdfstring{$\If$}{The idealizer} is Noetherian}
Let $f\in \kk[x,y]$ satisfy the hypotheses of Theorem \ref{fdeets}. In this section, we prove that $\If$ is noetherian on both sides. We start by showing that it is enough to prove that $\If$ is right noetherian, and then we show that we may check two separate cases determined by certain torsion properties.

We begin by setting up notation.
\begin{notation}\label{NOT}
Let $X = \mathbb{V}(f)$ for $f\in \kk[x,y]$ be a curve such that the normalisation map $\phi: \widetilde{X}\to X$ is injective . Also, let $D \coloneqq \DD(X)$ denote the ring of differential operators on $X$. Then, by \cite[Theorem 3.4]{SMST} $D$ is Morita equivalent to $A\coloneqq \DD(\widetilde{X})$, that is to say
$$D \cong \End_{A}(P)$$
where $P = \DD(\NX,X)$ is a $(D,A)$-bimodule which is finitely generated and projective on both sides.

We may form the right ideal $fA_2 \leq_{r} A_2$ and we denote \emph{the idealizer defined by $f$ in $A_2$} by $\If$.
\end{notation}
We begin this study on the noetherianity of $\If$ by proving a short result which allows us to only consider right noetherianity.
\begin{proposition}\label{LN}
To prove Theorem \ref{fdeets} it suffices to prove the right noetherian statement.
\end{proposition}
\begin{proof}
Suppose that we have proved that $\If$ is right noetherian for all $f \in \kk[x,y]$ such that $X = \mathbb{V}(f)$ satisfies the hypothesis of Theorem \ref{fdeets}. We show $fA_2f^{-1}\cap A_2 = \If $. Indeed,
\begin{align*}
     \If &= \{\theta \in A_2 \mid \theta fA_2 \subseteq fA_2\}\\
     & = \{\theta \in A_2 \mid f^{-1}\theta fA_2 \subseteq A_2\} \\
     & = \{\theta \in A_2 \mid \theta \subseteq fA_2f^{-1}\}\\
     & = A_2 \cap fA_2f^{-1}.
\end{align*}
Then,
$$\If \cong f^{-1}\If f = A_2 \cap f^{-1}A_2 f = \mathbb{I}_{A_2}(A_2f),$$
by symmetry. Now note that the hypotheses on the singularities of $X$ is left-right symmetric. Hence, by using a left-handed version of the proof of right noetherianity, $\If$ is left noetherian as required.
\end{proof}
By Propositon \ref{idealizer-diffops}, we have the following identification:
$$D=\DD(X) \cong \If/fA_2, $$
and we use this without reference throughout.
Proposition \ref{homs} states that $\If$ is right noetherian if and only if $\Hom_{A_2}(A_2/fA_2,A_2/J)$ is a finitely generated right $\If/fA_2 \cong D$-module for all right ideals $J\leq A_2$ which contain $fA_2$, we will also use this without reference throughout.

We start by showing that $A_2/IA_2$ is $1$-critical (in terms of (Gabriel-Rentschler) Krull dimension), where $I$ denotes any height $1$ prime ideal in $\kk[x,y]$.
\subsection{\texorpdfstring{$A_2/IA_2$}{A2/IA2} is 1-critical for all Height 1 Prime Ideals of \texorpdfstring{$\kk[x,y]$}{k[x,y]}}
In this section we show that $A_2/IA_2$ is $\kdim1$-critical for all height $1$ primes $I\lhd \kk[x,y]$, which we note are principal as $\kk[x,y]$ is a UFD. We do this in order to reduce to considering maximal right ideals $J\leq_r A_2$ which strictly contain $fA_2$.
\begin{theorem}\label{1critical}
Let $I$ be a height $1$ prime ideal in $\kk[x,y]$. Then $A_2/IA_2$ is $1$-critical.
\end{theorem}
Before we prove this, we provide some auxiliary definitions and results.
\begin{Definition}
A nonzero module $M_R$ over a ring $R$ is \emph{compressible} if for any submodule $N\leq M$ there exists a monomorphism $M\hookrightarrow N$.
\end{Definition}
\begin{example}\label{k/I comp}
We observe that $\kk[x,y]/I$ is compressible for all prime ideals $I$ since any pair of ideals are subisomorphic.
\end{example}
\begin{lemma}\label{comp-crit}\emph{\cite[Lemma $6.9.4$]{MR}}
A compressible module with Krull dimension is critical.
\end{lemma}
In light of this result, we show that $A_2/IA_2$ is compressible for any prime ideal $I\lhd \kk[x,y]$. We have the following result which shows how compressible modules behave over Ore extensions.
\begin{proposition}\label{McR-comp}\emph{\cite[Proposition $6.9.6$]{MR}}
Let $R$ be a ring, $\delta$ a derivation of $R$, and $S=R[x;\delta]$. If $M_R$ is compressible then $(M\otimes_{R}S)_S$ is compressible.
\end{proposition}
\begin{proof}[Proof of Theorem \ref{1critical}]
We begin by showing $\kdim(A_2/IA_2)=1$. For this, we simply observe that $A_2$ is 2-critical, thus $\kdim(A_2/IA_2)\leq 1$. Since $A_2/IA_2$ is not artinian (as an example consider an element $c \in \kk[x,y]\setminus{I}$ such that $c \notin \kk+I$ and $c$ is not a unit, then the descending chain $$A_2/IA_2\supsetneq (cA_2+IA_2)/IA_2 \supsetneq (c^2A_2+IA_2)/IA_2 \supsetneq \dots$$
is infinite), we have $\kdim(A_2/IA_2)=1$.

We must now show $A_2/IA_2$ is $1$-critical. Recall from Example \ref{k/I comp} that $\kk[x,y]/I$ is compressible for all prime ideals $I$. Hence, viewing $A_2$ as an iterated Ore extension, applying Proposition \ref{McR-comp} twice and Lemma \ref{comp-crit} we obtain the result.
\end{proof}
\begin{corollary}\label{Jfinleng}
If $J$ is a right ideal of $A_2$ which strictly contains $IA_2$, then $A_2/J$ has finite length.\qed
\end{corollary}
\subsection{Reduction to Simple Modules with a specific Torsion Property}
From here on, let $f\in \kk[x,y]$ be such that $X\coloneqq \mathbb{V}(f)$ satisfies the hypotheses of Notation \ref{NOT}. In this subsection we show that in our situation, in order to apply Proposition \ref{homs} it is not necessary to consider all right ideals $J$ which contain $fA_2$, rather just those $J$ which are maximal. We shall then see that considering a certain torsion property splits simple right $A_2$-modules into two types.
\begin{proposition}\label{homsequivs}
Assume Notation \ref{NOT}. Then the following are equivalent:
\begin{enumerate}
    \item[(1)] $\Hom_{A_2}(A_2/fA_2,A_2/J)$ is a finitely generated right $D$-module for all right ideals $J\leq_{r}A_2$ such that $fA_2\subsetneqq J$;
    \item[(2)] $\Hom_{A_2}(A_2/fA_2,A_2/J)$ is a finitely generated right $D$-module for all maximal right ideals $J\leq_{r}A_2$ which contain $fA_2$;
    \item[(3)] $\Hom_{A_2}(A_2/fA_2,N)$ is a finitely generated right $D$-module for all simple right $A_2$-modules $N$.
    \item[(4)] $\If$ is right noetherian.
\end{enumerate}
\end{proposition}
\begin{proof}
We observe the implications $(1)$ implies $(2)$ and $(3)$ implies $(2)$ are straightforward.
We show $(2)$ implies $(1)$ and $(3)$. Assume $(2)$ holds and let $N$ be a simple right $A_2$-module. Without loss of generality, suppose $\varphi\in \Hom_{A_2}(A_2/fA_2,N)$ is a nonzero homomorphism. Consider $0\neq m = \varphi(1+fA_2)$. As $N$ is simple $mA_2 = N$. Hence we have a surjective homomorphism $\psi:A_2\twoheadrightarrow N$ by $a\mapsto m.a$, and $N\cong A_2/\rann(m)$. Clearly $fA_2\subseteq\rann(m)\leq_{r}A_2$. Now $$\Hom_{A_2}(A_2/fA_2,N) \cong \Hom_{A_2}(A_2/fA_2,A_2/\rann(m))$$ where the right-hand side is a finitely generated right $D$-module by assumption. Hence $(3)$ holds.

Now we show that $(2)$ implies $(1)$. Let $J$ be a right ideal of $A_2$ which strictly contains $fA_2$. By Corollary \ref{Jfinleng}, $A_2/J$ has finite length, say $n$, and hence it has a composition series of the form
$$0\subseteq M_0/J\subseteq M_2/J \subseteq \dots \subseteq M_n/J = A_2/J$$
where $M_i\leq_{r}A_2$ are right ideals which contain $J$ and each $M_i/M_{i-1}$ is a simple right $A_2$-module. We show that $\Hom_{A_2}(A_2/IA_2, M_j/J)$ is a finitely generated right $\DD(X)$-module for all $j\geq 0$ by induction on $j$. If $j=0$ then $M_0/J$ is simple and the result follows since $(3)$ holds.

Now assume that $j>0$. Consider the short exact sequence of $A_2$-modules

$$0 \to M_{j-1}/J \to M_j/J \to M_j/M_{j-1} \to 0.$$
Applying $\Hom_{A_2}(A_2/fA_2,-)$ gives rise to an exact sequence of $D$-modules
\begin{align*}
\scalebox{0.91}{$0\to \Hom_{A_2}(A_2/fA_2,M_{j-1}/J)\xrightarrow{\alpha} \Hom_{A_2}(A_2/fA_2,M_j/J) \xrightarrow{\beta} \Hom_{A_2}(A_2/fA_2,M_{j}/M_{j-1})$}
\end{align*}
from which we extract the short exact sequence of $D$-modules
$$0 \to \Hom_{A_2}(A_2/fA_2,M_{j-1}/J) \to \Hom_{A_2}(A_2/fA_2,M_j/J) \to \im \beta \to 0. $$
As $\Hom_{A_2}(A_2/fA_2,M_{j-1}/J)$ is a finitely generated right $D$-module by induction, and $\im \beta$ is a submodule of $\Hom_{A_2}(A_2/fA_2,M_j/M_{j-1})$, which is finitely generated since $(3)$ holds and $D$ is right noetherian, $\Hom_{A_2}(A_2/fA_2,A_2/J)$ is finitely generated as required.

Finally we must show the equivalence of $(1)$ and $(4)$. By Proposition \ref{homs}, $(4)$ implies $(1)$. For the converse, Proposition \ref{homs} and the fact that $D$ is right noetherian (Corollary \ref{noethsimp}) show that we must check that  $\Hom_{A_2}(A_2/fA_2,A_2/fA_2)$ is a finitely generated right $D$-module. But this is clear since
$$\Hom_{A_2}(A_2/fA_2,A_2/fA_2)\cong \frac{(fA_2:fA_2)}{fA_2} = \frac{\If}{fA_2} \cong D.$$
\end{proof}
We now split simple $A_2$-modules into two cases.
\begin{Definition}
Let $M$ be a right $A_2$-module and let $\mathcal{C}$ be a right Ore set in $A_2$. Define 
$$t_{\mathcal{C}}(M) \coloneqq \{m \in M \mid m\cdot a = 0 \text{ for some } a \in \mathcal{C}\}$$
to be the set of $\mathcal{C}$-torsion elements of $M$. Observe that this is a submodule of $M$. If $t_{\mathcal{C}}(M) = 0$, we say $M$ is \emph{$\mathcal{C}$-torsionfree}, and if $t_{\mathcal{C}}(M) = M$, $M$ is \emph{$\mathcal{C}$-torsion}.
\end{Definition}
The following lemma splits simple $A_2$-modules into two natural cases. When $\mathcal{C} = \kk[y]^*$ we will abuse notation and say $M$ is $\kk[y]$-torsionfree or $\kk[y]$-torsion respectively; similarly if $\mathcal{C} = \kk[x]^*$.
\begin{lemma}\label{torN}
Let $M$ be a simple right $A_2$-module and consider the right Ore set $\mathcal{C} = \kk[y]^*$ or $\mathcal{C} = \kk[x]^*$. Then either $M$ is $\mathcal{C}$-torsion or $\mathcal{C}$-torsionfree.
\end{lemma}
\begin{proof}
By \cite[Lemma 4.21]{GW} $t_{\mathcal{C}}(M)$ is an $A_2$-submodule of $M$. Hence, by simplicity of $M$, either $t_{\mathcal{C}}(M) =0$ or $t_{\mathcal{C}}(M) =M$.
\end{proof}
We summarise our progress so far. In order to prove that $\If$ is right noetherian, we must show that  $\Hom_{A_2}(A_2/fA_2,A_2/J)$ is a finitely generated right $D$-module for all maximal right ideals $J$ which strictly contain $fA_2$; these fall into two cases: 
\begin{enumerate}
    \item[(1)] $A_2/J$ is $\kk[y]$-torsion and $\kk[x]$-torsion;
    \item[(2)] $A_2/J$ $\kk[y]$-torsionfree or $\kk[x]$-torsionfree.
\end{enumerate} We begin with case $(2)$.
\subsection{The Torsionfree Case}
In this section we show that if $A_2/J$ is $\kk[y]$- or $\kk[x]$-torsionfree, then the right $D$-module $\Hom_{A_2}(A_2/fA_2,A_2/J)$ is finitely generated for all maximal right ideals $J$ which strictly contain $fA_2$. We will see it is enough to prove this for the $\kk[y]$-torsionfree case.
\begin{lemma}\label{incl}
Let $J$ be a maximal right ideal of $A_2$ which strictly contains $fA_2$ and such that $A_2/J$ is $\kk[y]$-torsionfree. Let $g\in \kk[y]$ be the polynomial which defines the $y$-coordinates of the singular points of $X$. Note that $S = \{g^n\}_{n\geq 0}$ forms a right Ore set in $A_2$ (Lemma \ref{LcDiffs}) and so we may construct the localisation of $A_2$ by $S$:
$$A_2[g^{-1}] = A_2S^{-1}\coloneqq\{ab^{-1} \mid a \in A_2, \ b \in S\}.$$
Then the natural map of the following right $D$-modules 
$$\Hom_{A_2}(A_2/fA_2,A_2/J)\to \Hom_{A_2[g^{-1}]}(A_2[g^{-1}]/fA_2[g^{-1}],A_2[g^{-1}]/J[g^{-1}])$$
is injective.
\end{lemma}
\begin{proof}
Recall the definition of the ideal quotient of two right ideals $I,J$ in a ring $R$
$$(J:I) = \{r\in R \mid rI\subseteq J\}.$$
We have the following natural map of right $\If$-modules:
$$(J:fA_2) \xhookrightarrow{\iota} (J[g^{-1}]:fA_2[g^{-1}])\xtwoheadrightarrow{h} \frac{(J[g^{-1}]:fA_2[g^{-1}])}{J[g^{-1}]}$$
and let us consider the map $\pi = h\circ \iota$, so $\pi(a) = a+J[g^{-1}].$ Note that $J\subseteq \ker \pi$ and that the induced map 
$$\frac{(J:fA_2)}{J}\to \frac{(J[g^{-1}]:fA_2[g^{-1}])}{J[g^{-1}]}$$
intertwines with the isomorphisms from Proposition \ref{homsquotient} to give the map in the statement of the lemma. If we can show $\ker \pi = J$, then the result will follow by Proposition \ref{homsquotient}. To this end, let $m \in \ker \pi$. Then
\begin{align*}
    m \in J[g^{-1}] &\implies m = jg^{-n} \quad \textrm{ for some } j \in J \\
    & \implies mg^n \in J \implies m+J\in A_2/J \text{ is } \kk[y]\text{-torsion}.
\end{align*}
By our hypothesis, this implies $m\in J$.
\end{proof}
\begin{lemma}\label{localisationgood}
Let $X$ and $X'$ be birational, irreducible, affine curves which contain a common open subset $U$. Then the set of linear differential operators
$$\DD(X,U)= \DD(U)= \DD(X',U)$$
are equal.

Further, 
$$\DD(U)\otimes_{\DD(X')} \DD(X,X') \subseteq \DD(U).$$
\end{lemma}
\begin{proof}
Let the field of fractions associated to $\mathcal{O}(U)$ be denoted $K = \kk(U)$. Then we may identify $\DD(U)$ with its image in $\DD(K)$ as
$$\DD(U) = \{\theta\in \DD(K) \mid \theta*(\mathcal{O}(U))\subseteq \mathcal{O}(U)\}.$$
We may similarly identify $\DD(X,X')$ as 
$$\DD(X,X') = \{\theta\in \DD(K) \mid \theta*(\mathcal{O}(X))\subseteq \mathcal{O}(X')\}.$$
Since $\DD(X)$ is an HNP ring, $\DD(X,X')$ is a rank $1$ projective $\DD(X)$-module and in this proof we will identify $\DD(U)\otimes_{\DD(X')}\DD(X,X')$ with the subset $\DD(U)\DD(X,X')$ of the Goldie quotient ring $Q$. We will do this in future without comment.


As $\kk(X) = \kk(X') = \kk(U)$, we may use a standard argument in projective geometry to find $g \in \mathcal{O}(X)$ (respectively $g' \in \mathcal{O}(X')$) such that $\mathcal{O}(U) = \mathcal{O}(X)[g^{-1}]$ (respectively $\mathcal{O}(U) = \mathcal{O}(X')[g'^{-1}]$).

Then, as $\mathcal{O}(X)\subseteq \mathcal{O}(U)$, $\DD(U) \subseteq \DD(X,U)$. Now, let $D \in \DD(X,U)$ and observe that
$$D*(\mathcal{O}(U)) = D*(\mathcal{O}(X)[g^{-1}]) \subseteq D*(\mathcal{O}(X))[g^{-1}] \subseteq \mathcal{O}(U)[g^{-1}] = \mathcal{O}(U),$$
where the first containment follows by Lemma \ref{LcDiffs}. So $\DD(X,U) = \DD(U)$, and similarly, $\DD(X',U) = \DD(U)$.

Finally, 
$$\DD(U)\DD(X,X') *(\mathcal{O}(X)) = \DD(U) *(\mathcal{O}(X')) \subseteq \mathcal{O}(U),$$
and so $\DD(U)\DD(X,X') \subseteq \DD(X,U)  = \DD(U)$.
\end{proof}
\begin{proposition}\label{TFFG}
Let $J$ be a maximal right ideal of $A_2$ which strictly contains $fA_2$ such that $A_2/J$ is $\kk[y]$-torsionfree and let $g\in \kk[y]$ be the polynomial which defines the $y$-coordinates of the singular points of $X$. Then $$M\coloneqq \Hom_{A_2[g^{-1}]}(A_2[g^{-1}]/fA_2[g^{-1}],A_2[g^{-1}]/J[g^{-1}])$$ is a finitely generated right $D$-module. Consequently, $\Hom_{A_2}(A_2/fA_2,A_2/J)$ is a finitely generated right $D$-module.
\end{proposition}
\begin{proof}
Let the normalisation of $X$ be the curve $\NX$ and let $\phi:\NX\to X$ be the normalisation map which is bijective. We note that $\mathcal{O}(X)$ and $\mathcal{O}(\NX)$ may be viewed as subsets of $\kk(X)$, and hence $\DD(X)$ and $\DD(\NX)$ can be identified as subsets of $\DD(\kk(X))$. Let $A = \DD(\NX)$. We first claim that $M$ is a finite length right module over a localisation of $A$.

Indeed,
\begin{align*}
    M= \frac{(J[g^{-1}]:fA_2[g^{-1}])}{J[g^{-1}]} &= \frac{\mathbb{I}_{A_2[g^{-1}]}(fA_2[g^{-1}])+J[g^{-1}]}{J[g^{-1}]} \\&\cong \frac{\mathbb{I}_{A_2[g^{-1}]}(fA_2[g^{-1}])}{\mathbb{I}_{A_2[g^{-1}]}(fA_2[g^{-1}])\cap J[g^{-1}]},
\end{align*} 
where the second equality holds due to Proposition \ref{strucM} as $(f)\lhd \kk[x,y][g^{-1}]$ satisfy the conditions of Theorem \ref{McCfnoeth}. Further, due to Proposition \ref{strictc}, we note that $fA_2[g^{-1}]\subsetneqq \mathbb{I}_{A_2[g^{-1}]}(fA_2[g^{-1}])\cap J[g^{-1}]$ and hence 
$$\DD(\mathcal{O}(X)[\overline{g}^{-1}])\cong  \frac{\mathbb{I}_{A_2[g^{-1}]}(fA_2[g^{-1}])}{fA_2[g^{-1}]}\twoheadrightarrow M,$$
with nontrivial kernel, where $\overline{g}$ is the image of $g\in \kk[x,y]$ in $\mathcal{O}(X)$.

We now show that $C' \coloneqq \mathcal{O}(X)[\overline{g}^{-1}] \cong \mathcal{O}(\NX)[G^{-1}]$ where $G \in \mathcal{O}(\NX)$, whence we may conclude that $M$ is isomorphic to a proper factor of a localisation of $A$. Indeed, if $\phi: \NX\to X$ is our bijective normalisation map then, if $Y$ denotes the set of singular points in $X$ and $\widetilde{Y}$ the corresponding points in $\NX$, the restriction of $\phi$ to $X\setminus{Y}$ gives an isomorphism $X\setminus{Y}\cong \NX\setminus{\widetilde{Y}}$. Let $Z = \mathbb{V}(g)\cap X$, which is a finite set of points, and define $\widetilde{Z}\coloneqq\phi^{-1}(Z)$. Then $\phi$ also gives an isomorphism 
$$X\setminus{Z} \cong \NX\setminus{\widetilde{Z}}.$$
As in the proof of Lemma \ref{localisationgood}, we may find $G\in\mathcal{O}(\NX)$ which defines $\widetilde{Z}$. Then
$$C' = \mathcal{O}(X\setminus{Z})\cong \mathcal{O}(\NX\setminus{\widetilde{Z}}) =\mathcal{O}(\NX)[G^{-1}].$$
Hence $$\DD(C') = \DD(\mathcal{O}(\NX)[G^{-1}])= \DD(\mathcal{O}(\NX))[G^{-1}] = A[G^{-1}],$$
by Lemma \ref{LcDiffs}.

Let us now look more closely at $A[G^{-1}]$. We claim that this ring is $\kdim 1$-critical. Indeed, $A[G^{-1}]$ is a localisation of the HNP ring $A$ by \ref{smoothfacts}(b) , and is thus hereditary by Lemma \ref{LHisH}. Since it is not a division ring, $\kdim(A[G^{-1}])=1$ and, by \cite[Corollary 6.2.12]{MR}, any proper factor has $\kdim 0$ as required. Hence $M$ has finite length as a right $A[G^{-1}]$-module. Thus, to show that $M_{D}$ is finitely generated it suffices to show $S_{D}$ is finitely generated for all simple right modules $S$ of $A[G^{-1}] = \DD(C')\supseteq D.$ But simple $A[G^{-1}]$-modules are holonomic so, by Theorem \ref{Bern}, $S_{A}$ is holonomic and thus is finitely generated. (Note that we are using the fact that $\NX$ is nonsingular here.)

Now recall by Theorem \ref{SS} that $D$ and $A$ are Morita equivalent where the progenerator is
$$P \coloneqq \DD(\NX,X) = \{\theta \in A\mid \theta*(\mathcal{O}(\NX))\subseteq \mathcal{O}(X)\}.$$
We show $A[G^{-1}]\otimes_{D} P \cong A[G^{-1}]$ as $A$-modules. We note that $A[G^{-1}]\otimes_{D} P \subseteq A[G^{-1}]$ by Lemma \ref{localisationgood}. We now show the reverse containment. Observe that $P\otimes_{A}P^*$ (resp. $P^*\otimes_{D}P$) is a nonzero two sided ideal of $D$ (resp. $A$), and hence $P\otimes_{A}P^*=D$ (resp. $P^*\otimes_{D}P = A$) by simplicity, $A$ by \ref{smoothfacts}(a) and $D$ by Corollary \ref{noethsimp}. Thus
$$A[G^{-1}] = \left(A[G^{-1}]\otimes_{D} P\right)\otimes_{A}P^* \subseteq A[G^{-1}]\otimes_{A} P^* \subseteq A[G^{-1}],$$
where the last containment follows by Lemma \ref{localisationgood} as $P^* = \DD(X,\NX)$.
Hence,
\begin{equation*}
    \left(A[G^{-1}]\otimes_{D} P\right)\otimes_{A}P^*  =  A[G^{-1}]\otimes_{D} P^* ,
\end{equation*}
and as $P^*$ is a progenerator, applying $-\otimes_{D}P$ to both sides gives $A[G^{-1}]\otimes_{A} P \cong A[G^{-1}]$.

We claim that $S_D\otimes_D P \cong S_A$. By Proposition \ref{GOOD}, we see that $S \cong S'\otimes_{A} A[G^{-1}]$ where $S'$ is a finitely generated simple right $A$-module. Now,
\begin{align*}
    S\otimes_{D}P &= S'\otimes_{A} \left(A[G^{-1}]\otimes_{D} P \right)\\
    &\cong S' \otimes_{A} A[G^{-1}] = S_{A}.
\end{align*}
As $S_{A}$ is finitely generated, and finite generation is preserved by Morita equivalence, $S_{D}$ is finitely generated as required.
\end{proof}
We have shown that if $A_2/J$ is $\kk[y]$-torsionfree where $J$ is a maximal right ideal which contains $fA_2$, then $\Hom_{A_2}(A_2/fA_2,A_2/J)$ is a finitely generated right $D$-module. If $A_2/J$ were instead a $\kk[x]$-torsionfree module, an identical method would work.
\subsection{The Torsion Case}
We must show that $\Hom_{A_2}(A_2/fA_2,A_2/J)$ is a finitely generated right $D$-module for all maximal right ideals $J$ of $A_2$ which strictly contain $fA_2$ and such that $A_2/J$ is $\kk[x]-$ and $\kk[y]$-torsion. We shall see that we may effectively reduce to showing that the right $D$-module:
\begin{equation*}
    \mathcal{M} = \frac{((x,y)A_2:fA_2)}{(x,y)A_2}
\end{equation*}
is finitely generated.
We begin by defining a special type of filtration on a ring.
\begin{Definition}
Let $R = \bigcup_{n\geq 0} R(n)$ be a filtered ring. We say that $R$ is \emph{strongly filtered} if
$$R(i)R(j) = R(i+j)$$
for all $i,j \in \NN$.
\end{Definition}
\begin{Definition}\label{qfilt}
Let $T$ be a module over a ring $R$ with a strong filtration $R = \bigcup R(i)$. We say $T$ is \emph{quasi-filtered} if $T$ may be written as a union of nested finite-dimensional vector spaces $T = \bigcup_{i\in \NN}T(i)$ such that there exists an $a \in \NN$ such that 
$$T(i)R(1)\subseteq T(i+a)$$ for all $i\in \NN$. Further, we say $T$ has \emph{generalised linear growth} if there exist $c\in \RR_+$, $d\in \RR$ such that
$$\dim_{\kk}T(i) \leq ci+d,$$
for all $i\in \NN$.
\end{Definition}
We now prove the following theorem which shows that modules with suitably bounded growth over strongly filtered simple rings will always be finitely generated.
\begin{theorem}\label{glgfg}
Let $A$ be a strongly filtered infinite dimensional simple ring. Then any quasi-filtered right $A$-module with generalised linear growth is finitely generated.
\end{theorem}
Before we prove this theorem, we first prove some auxiliary lemmata.
\begin{lemma}\label{lbound}
Let $A$ be a strongly filtered simple infinite dimensional ring and let $U$ be a nonzero finite dimensional subspace of a right $A$-module. Then
$$\dim_{\kk}UA(m)\geq m \quad \forall m\in \NN.$$
\end{lemma}
\begin{proof}
We begin by noting that $UA(m)\supsetneqq UA(m-1)$ for all $m\in \NN$. Indeed, if this were not the case then $UA = UA(m)$ would be a nonzero finite dimensional $A$-module since $A$ is strongly filtered, which is impossible.
Hence
\begin{gather*}
\dim_{\kk}UA(m)\geq \dim_{\kk}UA(m-1)+1 \geq \dots \geq \dim_{\kk}UA(0)+m\geq m. \qedhere
\end{gather*}
\end{proof}
\begin{lemma}\label{AC}
Let $A$ be a strongly filtered simple infinite dimensional ring and assume that $N$ is not a finitely generated right $A$-module. Then we may form a infinite ascending chain of finite dimensional subspaces $\{U_i\}_{i\geq 1}$ of $N$ such that
$\dim_{\kk}U_{i}A(m) \geq im$ for all $i,m\geq 1$.
\end{lemma}
\begin{proof}
We proceed by induction on $i\geq 1$, noting that the preceding lemma gives the base case $i=1$. Now we assume that there exists a finite dimensional $U_i$ such that $\dim_{\kk}U_iA(m)\geq im$ for all $m\in \NN$. Now, since $N$ is not finitely generated, $N\neq U_iA$, so there exists a finite dimensional $U_{i+1}\supsetneq U_i$ with $U_{i+1}\nsubseteq U_iA$. Let $\overline{U_{i+1}}$ denote the image of $U_{i+1}$ in $N/U_iA$. Consider the factor module $$\overline{U_{i+1}A} \coloneqq U_{i+1}A/U_iA = \overline{U_{i+1}}A.$$
Then
\begin{align*}
    \dim_{\kk}U_{i+1}A(m)& \geq \dim_{\kk}U_{i}A(m)+\dim_{\kk}\overline{U_{i+1}}A(m) \\
    &\geq im+m
\end{align*}
where the last inequality follows from Lemma \ref{lbound}.
\end{proof}
\begin{lemma}\label{upbound}
Let $A$ be a strongly filtered simple infinite dimensional ring and let $N$ be a quasi-filtered right $A$-module with generalised linear growth. Then any strictly ascending chain of finite-dimensional subspaces $\{U_i\}_{i\in \NN}$ \ such that $\dim_{\kk}U_iA(m)\geq im$ for all $i,m\in \NN$ is finite.
\end{lemma}
\begin{proof}
Note that for the $U_i$ as in the statement of the lemma there exists a $u_i\in \NN$ such that $U_i\subseteq N(u_i)$, since they are finite dimensional. Further, for all $m\in \NN$,
$$U_iA(m) \subseteq N(u_i)A(m) = N(u_i)A(1)^m\subseteq N(u_i+ma)$$
as $N$ is quasi-filtered, which implies that 
$$im \leq \dim_{\kk}(U_iA(m))\leq \dim_{\kk}N(u_i+ma) \leq c(u_i+ma)+d,$$
and hence $i\leq ca$.
\end{proof}
\begin{proof}[Proof of Theorem \ref{glgfg}]
Assume for contradiction that $N_{A}$ is a quasi-filtered right module with generalised linear growth which is not finitely generated. Then by Lemma \ref{AC} we may form an ascending chain of subspaces $\{U_i\}_{i\in\NN}$ such that $\dim_{\kk}U_iA(m)\geq im$ for all $i,m\geq 1$. But this contradicts Lemma \ref{upbound}. 
\end{proof}
Recall Notation \ref{NOT}. We wish to apply Theorem \ref{glgfg} to the $D = \DD(X)$-module $$\mathcal{M} = \frac{((x,y)A_2:fA_2)}{(x,y)A_2}.$$

We now show that in our circumstances we have two filtrations on $D = \DD(X)$ which interact well together. For the first filtration, recall the Bernstein filtration $\mathcal{B} = \bigcup_{i\geq 0}\mathcal{B}_i$ on the second Weyl algebra $A_2$ defined by 
$$\mathcal{B}_n = \Span_{\kk}\{x^iy^j\dx^k\dy^{\ell}\mid i+j+k+\ell\leq n\}.$$
That is to say, $\mathcal{B}_n$ is the $\kk$-subspace of $A_2$ spanned by monomials in $x,y,\dx,\dy$ of degree $\leq n$. Recall also, by Proposition \ref{idealizer-diffops}, we have that $D \cong\If/fA_2$, and we use this to define a filtration on $D$. First we view $\If$ as a subring of $A_2$ with the Bernstein filtration $\mathcal{B}$ and define a filtration on $\If$ by setting $\mathcal{G}_m = \If\cap \mathcal{B}_m$. Then we define a filtration on the ring $\If/fA_2$ by 
$$\mathcal{F}_m \coloneqq \left(\If/fA_2\right)_m = \frac{\mathcal{G}_m+fA_2}{fA_2}.$$
By abuse of notation, we will refer to this as the \emph{Bernstein filtration on $D$.} We use subscripts here to clearly distinguish which filtration on $D$ we are using.

For the second filtration, by Corollary \ref{noethsimp}, $D$ has a strong filtration which we denote
$$D(m) = V^m$$
where $V$ is a finite generating set for $D$ containing $1$.

We claim these two filtrations interact well together, which we formalise in the following proposition.
\begin{lemma}\label{goodstrongfilt}
Recall Notation \ref{NOT}. Consider the Bernstein filtration on $D$ and let $M$ be a filtered $D$-module with respect to the Bernstein filtration. Choose a finite generating set $V$ for $D$ containing $1$ and define a strong filtration on $D$ by $D(m) = V^m$. Then $M$ is quasi-filtered in the sense of Definition \ref{qfilt}.
\end{lemma}
\begin{proof}
Note that each filtered piece of $D$ under the strong filtration is finite dimensional, hence there exists $a\in\NN$ such that
$$D(1)\subseteq \mathcal{F}_{a}.$$
Hence
$$M(i)D(1)\subseteq M(i)\mathcal{F}_{a}\subseteq M(i+a),$$
as required.
\end{proof}
\begin{proposition}\label{torsionproved}
Let $\gamma,\delta \in \kk$. The right $D$-module
$$\mathcal{M} = \frac{((x-\gamma,y-\delta)A_2:IA_2)}{(x-\gamma,y-\delta)A_2}$$
is finitely generated.
\end{proposition}
\begin{proof}
We first observe that for any point $(\gamma, \delta) \in \mathbb{A}^2$, we have the following isomorphism:
$$\frac{((x-\gamma,y-\delta)A_2:fA_2)}{(x-\gamma,y-\delta)A_2} \cong \frac{((x',y')A_2:\hat{f}A_2)}{(x',y')A_2},$$
where $\hat{f}\in \kk[x,y]$ defines a curve with the same properties as $X$. Hence without loss of generality it suffices to prove that $\mathcal{M}$ is finitely generated for $\gamma = \delta = 0$.

We use Theorem \ref{glgfg}. We define a filtration on $\mathcal{M}$ by first denoting $$\Gamma_n \coloneqq ((x,y)A_2:fA_2)\cap \mathcal{B}_n$$ and then constructing
$$\mathcal{M}(n) \coloneqq \frac{\Gamma_n+(x,y)A_2}{(x,y)A_2}.$$
We again abuse notation and refer to this as the \emph{Bernstein filtration on $\mathcal{M}$}. Now we note that $\Gamma_n \cdot \mathcal{G}_m \subseteq \Gamma_{n+m}$ by definition. Then
\begin{align*}
\mathcal{M}(n)\cdot \mathcal{F}_m & = \frac{\Gamma_n+(x,y)A_2}{(x,y)A_2} \cdot \frac{\mathcal{G}_m+fA_2}{fA_2} \\
& \subseteq \frac{\Gamma_n\cdot \mathcal{G}_m+(x,y)A_2}{ (x,y)A_2}\\
& \subseteq \frac{\Gamma_{n+m}+(x,y)A_2}{ (x,y)A_2} = \mathcal{M}(n+m).
\end{align*}
So $\mathcal{M}$ is a filtered $D$-module when $D$ has the induced Bernstein filtration. Now we show $\mathcal{M}$ has generalised linear growth.

We will now use some standard linear algebra techniques related to linear independence of equations to bound the dimension of $\mathcal{M}(n)$ where $n\gg0$. We begin by describing the important sets we will use. Firstly, given a polynomial $f = \sum_{\alpha,\beta}f_{\alpha \beta}x^{\alpha}y^{\beta} \in \kk[x,y]$ we define the set
$$\supp(f) \coloneqq \{(i,j) \mid f_{ij}\neq 0\}\subseteq \NN^2.$$
Then we note
$$\supp(fx^ay^b) = \supp(f)+(a,b).$$
Write $A_2/(x,y)A_2 = \kk[\dx,\dy]$, which we filter by setting $\deg\dx = \deg\dy = 1$; this induces the Bernstein filtration on $\mathcal{M} \subseteq A_2/(x,y)A_2$. That is to say, we view the filtered pieces of $\mathcal{M}$ as: 
$$\mathcal{M}(n) =\{\theta \in \kk[\dx,\dy]_{\leq n}\mid \theta * fx^iy^j\in (x,y)\quad \forall i,j \in \NN\}.$$
We also define 
$$\Delta_n = \{(i,j)\in \mathbb{N}^2\mid i+j\leq n\},$$
Let $N = \max_{\text{totaldeg}}(f)$. Then, if $n> N$, we see
$$\Delta_{n-N} = \{(i,j)\in \NN^2 \mid \supp(f)+(i,j)\subseteq \Delta_n\}.$$
Fix $n>N$. We first claim that $$\dim_{\kk}\mathcal{M}(n)\leq |\Delta_n|-|\Delta_{n-N}|.$$
Put a deg-lex ordering on pairs $(i,j)\in \NN^2$, denoted $\prec$, where $(a,b)\prec (c,d)$ if $a+b<c+d$ or, if there is equality of total degree, and $(a,b)$ is lexicographically less than $(c,d)$. Note that
\begin{equation}\label{deglex}
    (a,b)\prec (c,d) \iff (a+i,b+j)\prec (c+i,d+j).
\end{equation}
Now we show how to give an upper bound for $\dim_{\kk} \mathcal{M}(n)$. 

Let $\theta = \sum_{0\leq p+q\leq n}a_{pq}\dx^p\dy^q $ be a general element of $\kk[\dx,\dy]_{\leq n}$, where the $a_{pq}$ are indeterminates. Also let 
\begin{equation}\label{lambda}
\Phi_{ij} \coloneqq \left(\theta * fx^iy^j\right) (0,0) =  \sum_{(\alpha,\beta)\in \supp f}a_{\alpha+i,\beta+j}(\alpha+i)!(\beta+j)!f_{\alpha \beta},
\end{equation}
which is a linear combination of the variables $a_{pq}.$ Then $\theta\in \mathcal{M}(n)$ if and only if $\Phi_{ij} = 0$ for all $i,j \in \NN$. When $(i,j)\in \Delta_{n-N}$, $\Phi_{ij}$ is nonzero as an element of the set $\Span_{\kk}\{a_{pq}\mid (p,q) \in \Delta_n\}$.

We extend the monomial ordering on $\Delta_n$ to the $a_{pq}$ by defining $a_{\alpha \beta}\prec a_{\alpha'\beta'}$ if and only if $(\alpha, \beta)\prec (\alpha', \beta')$. Further, we denote the element of $\NN^2$ associated to the leading term of $\Phi_{ij}$ by $\overline{\Phi_{ij}}$, so $\overline{\Phi_{i j}} \coloneqq (k,\ell)$ where $a_{k\ell} = \lt_{\prec}(\Phi_{ij})$.
Then we note that 
$$\overline{\Phi_{ij}} = (i,j)+\overline{\Phi_{00}},$$
by \eqref{lambda}.

Let $\{\beta_{ij}\}_{0\leq i+j \leq n}\in \kk$ be not all zero and let 
$$M = \max_{\prec}\{(i,j)\mid \beta_{ij}\neq 0\} = (i_M,j_M).$$
If $(i,j)\prec M$, then $\overline{\Phi_{ij}}\prec M+\overline{\Phi_{00}}$ by \eqref{deglex}. Consider the linear combination of the $\Phi_{ij}$, $$\Phi \coloneqq \sum_{0\leq i+j\leq n}\beta_{ij}\Phi_{ij},$$
viewed as a linear expression in the $a_{pq}$. Then $\overline{\Phi} = \overline{\Phi_{i_Mj_M}} = M+\overline{\Phi_{00}}$ and the leading coefficient of $\Phi$ is $\lambda \beta_{i_Mj_M}$ where $\lambda$ is a nonzero coefficient coming from \eqref{lambda}. Hence, $\Phi =0$ implies $\beta_{i_Mj_M}=0$, a contradiction.

Thus $\{\Phi_{ij}\mid (i,j)\in \Delta_{n-N}\}$ are linearly independent and the elements of $\Delta_{n-N}$ give linearly independent constraints on $\mathcal{M}(n)$.
Hence 
$$\dim_{\kk}\mathcal{M}(n)\leq |\Delta_n|-|\Delta_{n-N}|,$$
as claimed.

We now observe that 
$$\dim_{\kk}\mathcal{M}(n)\leq |\Delta_n|-|\Delta_{n-N}| = \frac{n(n+1)}{2}-\frac{(n-N)(n-N+1)}{2}$$
which is linear in $n$, and hence $\mathcal{M}$ has generalised linear growth.

Put a strong filtration $D = \bigcup_{n\geq 0}D(n)$ on $D$ as described in Lemma \ref{goodstrongfilt}. By Lemma \ref{goodstrongfilt}, there exists an $a\in \NN$ such that $$\mathcal{M}(n)D(1)\subseteq \mathcal{M}(n+a)$$
for all $n\geq 0$. Thus, by Theorem \ref{glgfg}, $\mathcal{M}_D$ is finitely generated.
\end{proof}
We now complete the proof of the right noetherianity of $\If$. The following auxiliary lemma aids us. 
\begin{lemma}\label{compS}
Let $S$ be a simple $\kk[x]$-torsion right $A_1$-module, then 
$$S \cong A_1/(x-a)A_1$$
for some $a \in \kk.$
\end{lemma}
\begin{proof}
As $S$ is a simple module with $\kk[x]$-torsion, that implies that $S \cong A_1/K$ where $K$ is a maximal right ideal of $A_1$ which contains a polynomial $p$ in $x$ (the annihilator of $1$). Let us first consider a composition series for the right module $A_1/pA_1$.

Since $\kk$ is algebraically closed, we may factorise $p$ into linear factors:
$$ p = \prod_{i=1}^n(x-a_i)$$
and denote
$$p_j = \prod_{i=j}^n(x-a_i),$$
for $j = 1,\dots,n$.

Observe that
$$p_jA_1/p_{j+1}A_1 \cong A_1/(x-a_{j+1})A_1,$$
which is a simple right $A_1$-module. Hence the following induces a composition series for $A_1/pA_1$:
$$
pA_1\subseteq p_2A_1\subseteq \dots \subseteq p_nA_1\subseteq A_1.$$
We finish the proof by noting that, as $pA_1\subseteq K$,
$$A_1/pA_1\twoheadrightarrow A_1/K,$$
and hence, by the Jordan-H\"older theorem, $A_1/K$ must be isomorphic to a composition factor of $A_1/pA_1$. That is to say
$$S \cong A_1/K \cong A_1/(x-a)A_1$$
for some $a\in A_1$.
\end{proof}
\begin{proposition}\label{STPstar} Assume Notation \ref{NOT}. Suppose that $J\leq_{r} A_2$ is maximal and contains $fA_2$, and that $A_2/J$ is $\kk[y]$- and $\kk[x]$-torsion. Then there exist $a,b \in \kk$ such that
$$A_2/J \cong A_2/(x-a,y-b)A_2.$$
\end{proposition}
\begin{proof}
Since $A_2/J$ is $\kk[x]$-torsion, this implies that there exists a nonzero polynomial $g \in J \cap \kk[x]$ (the annihilator of the identity). By \cite[Proposition 5.1]{BVO}, $A_2/J \cong S_1\otimes_{\kk}S_2$, where $S_1$ is a simple $\kk[x,\dx]$-module and $S_2$ is a simple $\kk[y,\dy]$-module, and there exists a nonzero $s \in S_1$ such that $sg = 0$. We note, by the same reasoning as Lemma \ref{torN}, $S_1$ is either $\kk[x]$-torsion or $\kk[x]$-torsionfree. As $sg=0$, $S_1$ must be $\kk[x]$-torsion. Applying Lemma \ref{compS}, we may conclude $S_1\cong A_1/(x-a)A_1$ for some $a\in \kk$. Similarly, we may conclude that $S_2 \cong A_1/(y-b)A_1$ and thus $A_2/J \cong S_1\otimes_{\kk} S_2\cong A_2/(x-a,y-b)A_2$.
\end{proof}
We now prove Theorem \ref{fdeets}.
\begin{proof}[Proof of Theorem \ref{fdeets}]
By Proposition \ref{homsequivs} to show that $\If$ is right noetherian we must show that $\Hom_{A_2}(A_2/fA_2,A_2/J)$ is a finitely generated right $D$-module for all maximal right ideals $J\leq_{r}A_2$ which strictly contain $fA_2$.

Lemma \ref{torN} shows that there are two cases:
\begin{enumerate}
    \item $A_2/J$ is $\kk[x]$- or $\kk[y]$-torsionfree; or
    \item $A_2/J$ is $\kk[x]$- and $\kk[y]$-torsion.
\end{enumerate}

Proposition \ref{TFFG} covers case $(1)$. For case $(2)$, Proposition \ref{STPstar} shows that we need only consider right ideals $J$ of the form $(x-a,y-b)A_2$ where $(a,b)\in \kk^2$ and Proposition \ref{torsionproved} shows that 
$$\Hom_{A_2}(A_2/fA_2,A_2/(x-a,y-b)A_2) \cong \frac{((x-a,y-b)A_2:IA_2)}{(x-a,y-b)A_2},$$
is finitely generated as required. Then, by Proposition \ref{LN}, $\If$ is left and right noetherian.
\end{proof}
\bibliographystyle{siam}
\bibliography{bibliography}

\begin{thebibliography}{10}

\bibitem{BVO}
{\sc V.~Bavula and F.~Van~Oystaeyen}, {\em Simple holonomic modules over the
  second {W}eyl algebra {$A_2$}}, Advances in Mathematics, 150 (2000),
  pp.~80--116.

\bibitem{B}
{\sc J.~Bernstein}, {\em Algebraic theory of {$D$}-modules}.
\newblock
  \url{http://math.columbia.edu/~khovanov/resources/Bernstein-dmod.pdf}.

\bibitem{bernstein1972analytic}
\leavevmode\vrule height 2pt depth -1.6pt width 23pt, {\em The analytic
  continuation of generalized functions with respect to a parameter},
  Funktsional'nyi Analiz i ego Prilozheniya, 6 (1972), pp.~26--40.

\bibitem{bjorkdiff}
{\sc J.-E. Bj{\"o}rk}, {\em Rings of Differential Operators}, vol.~21,
  North-Holland, 1979.

\bibitem{Elliott}
{\sc C.~Elliott}, {\em {$D$}-modules}.
\newblock
  \url{http://math.columbia.edu/~khovanov/resources/Bernstein-dmod.pdf}.

\bibitem{GW}
{\sc K.~Goodearl and R.~Warfield}, {\em An Introduction to Noncommutative
  Noetherian Rings}, vol.~61, Cambridge university press, 2004.

\bibitem{Groeth}
{\sc A.~Grothendieck}, {\em {\'E}l{\'e}ments de g{\'e}om{\'e}trie
  alg{\'e}brique}, Bulletin of the American Mathematical Society, 67 (1961),
  pp.~239--246.

\bibitem{Mccaffrey}
{\sc D.~McCaffrey}, {\em Idealisers and rings of differential operators},
  Communications in Algebra, 41 (2013), pp.~675--702.

\bibitem{MR}
{\sc J.~McConnell and J.~Robson}, {\em Noncommutative Noetherian Rings},
  vol.~30, American Mathematical Soc., 2001.

\bibitem{Muh}
{\sc J.~Muhasky}, {\em The differential operator ring of an affine curve},
  Transactions of the American Mathematical Society, 307 (1988), pp.~705--723.

\bibitem{Resco1980}
{\sc R.~Resco}, {\em {K}rull dimension of noetherian algebras and extensions of
  the base field}, Communications in Algebra, 8 (1980), pp.~161--183.

\bibitem{Resco1983}
\leavevmode\vrule height 2pt depth -1.6pt width 23pt, {\em Affine domains of
  finite {G}elfand-{K}irillov dimension which are right, but not left,
  noetherian}, Bulletin of the London Mathematical Society, 16 (1984),
  pp.~590--594.

\bibitem{Robson1972}
{\sc J.~Robson}, {\em Idealizers and hereditary noetherian prime rings},
  Journal of Algebra, 22 (1972), pp.~45--81.

\bibitem{Rog}
{\sc D.~Rogalski}, {\em Idealizer rings and noncommutative projective
  geometry}, Journal of Algebra, 279 (2004), pp.~791--809.

\bibitem{SMST}
{\sc S.~P. Smith and J.~T. Stafford}, {\em Differential operators on an affine
  curve}, Proceedings of the London Mathematical Society, 3 (1988),
  pp.~229--259.

\end{thebibliography}
\end{document}